\documentclass[a4paper,11pt]{amsart}
\usepackage{amssymb,mathtools}
\usepackage{enumerate}
\usepackage[T1]{fontenc}
\usepackage{lmodern}
\usepackage{textcomp}
\usepackage{microtype}
\usepackage[english]{babel}
\usepackage{xcolor,colortbl}
\usepackage{hyperref}
\newtheorem{theorem}{Theorem}[section]
\newtheorem{lemma}[theorem]{Lemma}
\newtheorem{proposition}[theorem]{Proposition}
\newtheorem{corollary}[theorem]{Corollary}
\theoremstyle{definition}
\newtheorem{definition}[theorem]{Definition}
\newtheorem{model}{Hypothesis}
\theoremstyle{remark}
\newtheorem{remark}[theorem]{Remark}
\newtheorem{example}[theorem]{Example}
\newenvironment{hypothesis}[1]{
	\renewcommand\themodel{#1}
	\model
}{\endmodel}
\numberwithin{equation}{section}
\newcommand{\RR}{\mathbb{R}}
\newcommand{\CC}{\mathbb{C}}
\newcommand{\NN}{\mathbb{N}}
\newcommand{\ZZ}{\mathbb{Z}}
\newcommand{\cE}{\mathcal{E}}
\newcommand{\cJ}{\mathcal{J}}
\newcommand{\cI}{\mathcal{I}}
\newcommand{\cB}{\mathcal{B}}
\newcommand{\cU}{\mathcal{U}}
\newcommand{\eps}{\varepsilon}
\newcommand{\euler}{\mathrm{e}}
\newcommand{\indic}{\mathbf{1}}
\DeclareMathOperator{\Ran}{Ran}
\DeclareMathOperator{\diam}{diam}
\newcommand*\Diff[1]{\mathop{}\!\mathrm{d}#1}
\DeclarePairedDelimiter{\abs}{|}{|}
\DeclarePairedDelimiter{\norm}{\lVert}{\rVert}
\newcommand{\bes}{\begin{equation*}}
\newcommand{\ees}{\end{equation*}}
\newcommand{\be}{\begin{equation}}
\newcommand{\ee}{\end{equation}}
\newcommand{\eqs}[1]{\begin{align*}#1\end{align*}}

\title[Uncertainty principle for Hermite functions]
{Uncertainty principle for Hermite functions and null-controllability with sensor sets of decaying density}
\hypersetup{
	pdftitle={Uncertainty principle for Hermite functions and null-controllability with sensor sets of decaying density},
	pdfauthor={Alexander Dicke, Albrecht Seelmann, Ivan Veseli\'c},
	hidelinks
}
\author[A.~Dicke]{Alexander Dicke}
\author[A.~Seelmann]{Albrecht Seelmann}
\author[I.~Veseli\'c]{Ivan Veseli\'c}
\address[A.D., A.S., I.V.]{
	Technische Univer\-si\-t\"at Dortmund,
	Germany
}
\urladdr{https://www.mathematik.tu-dortmund.de/lsix/}

\email{adicke.math@gmail.com}
\email{albrecht.seelmann@mathematik.tu-dortmund.de}
\email{ivan.veselic@mathematik.tu-dortmund.de}
\begin{document}
%
%

\begin{abstract}
	We establish a family of uncertainty principles for finite linear combinations of Hermite functions.
	More precisely, we give a geometric criterion on a subset $S\subset \RR^d$ ensuring that the $L^2$-seminorm associated to $S$
	is equivalent to the full $L^2$-norm on $\RR^d$ when restricted to the space of Hermite functions up to a given degree.
	We give precise estimates how the equivalence constant depends on this degree and on geometric parameters of $S$.
	From these estimates we deduce that the parabolic equation whose generator is the harmonic oscillator is null-controllable from $S$.

	In all our results, the set $S$ may have sub-exponentially decaying density and, in particular, finite volume.
	We also show that bounded sets are not efficient in this context.
\end{abstract}
\maketitle

%
%

\section{Introduction}

We establish uncertainty relations for functions in the linear span $\cE_N$ of Hermite functions up to degree $N\in\NN$.
More precisely, we show that their restriction to a properly chosen subset $S\subset \RR^d$
has equivalent norm to the function on all of $\RR^d$ with explicitly spelled out constant.
Such estimates bear various names depending on the area of analysis where they appear.
For instance, \emph{quantitative unique continuation estimate}, see, e.g.,~\cite{RousseauL-12,LogunovM-20},
\emph{uncertainty principle}, see, e.g.,~\cite{StollmannS-21}, or \emph{spectral inequality} (in the context of control theory), see, e.g.,~\cite{RousseauL-12,LaurentL-21}.
It is also closely related to the notion of \emph{vanishing order}, see, e.g.,~\cite{DonnellyF-88,LaurentL-21}, and  \emph{annihilating pairs} in Fourier analysis,
see for instance \cite{HavinJ-94,BeauchardJPS-21,EgidiNSTTV-20}.

Our estimate improves upon recent results from \cite{BeauchardJPS-21, MartinPS-22} in several aspects simultaneously:
\begin{enumerate}[(i)]
	\item The set $S$ is allowed to become sparse near infinity and may even have finite Lebesgue measure,
	\item the gaps or holes in $S$ are allowed to increase near infinity in a very general manner.
\end{enumerate}
Although this means that $S$ may be quite sparse, we show that there are still constants $d_0,d_1 > 0$ and $\zeta \in (0,1)$
such that for all $N\in\NN$ we have
\[
	\norm{f}_{L^2(\RR^d)}^2\leq d_0\euler^{d_1N^\zeta}\norm{f}_{L^2(S)}^2,\quad f\in\cE_N.
\]
Since $\Ran \indic_{(-\infty,2N+d]}(-\Delta+|x|^2) = \cE_N$, this is a so-called \emph{spectral inequality} for the harmonic oscillator $-\Delta+|x|^2$.
An important feature here is that the exponent is sublinear in $N$.
This allows for an application to null-controllability for the Hermite semigroup, see Section~\ref{sec:application} below.
For general lower semibounded, self-adjoint operators $H$ on $L^2$, such spectral inequalities take the form
\[
	\norm{f}_{L^2}^2\leq d_0\euler^{d_1\lambda^\zeta}\norm{f}_{L^2(S)}^2,\quad f\in \indic_{(-\infty,\lambda]}(H),
\]
with constants $d_0,d_1 > 0$ and $\zeta \in (0,1)$ depending on $H$ and the subset $S$ of the domain under consideration.
If such an inequality holds, we call the set $S$ \emph{efficient} (for $H$).
Note that if a set $S$ is not efficient in this sense, this does a priori \emph{not} imply that the corresponding parabolic equation is not null-controllable from $S$.

While our results belong to the realm of harmonic analysis,
they have a number of applications in various areas of the theory of partial differential equations and operators.
In this paper, we focus on just one of the applications and present it in Section~\ref{sec:application},
namely the above mentioned null-controllability for the parabolic harmonic oscillator evolution.
A wider range of applications in control theory will be presented in a forthcoming project.

Two questions arising from earlier papers triggered our research:

\subsection*{Fast decay of Hermite functions}
For the heat equation on $\RR^d$ it has recently been established independently in \cite{EgidiV-18} and \cite{WangWZZ-19} that a sensor set $S$
ensures null-controllability \emph{if and only if} it is \emph{thick} in the sense of Definition~\ref{def:thick} below.
In \cite{BeauchardJPS-21} it was shown that thickness is \emph{sufficient} for the null-controllability of the Hermite semigroup.
However, the quadratic potential of the harmonic oscillator enforces fast decay of any eigenfunction.
This begs the following question: \emph{Is it possible to control the Hermite semigroup from a sensor set $S$ that has finite measure?}

\subsection*{Reconciling phenomena on bounded and unbounded domains}
On bounded domains any set of positive Lebesgue measure can serve as a sensor or control
set for observability or null-controllability, respectively, of the heat equation, 
see, e.g.,~\cite{FursikovI-96,LebeauR-95,ApraizEWZ-14}.
As mentioned above, for the heat equation on $\RR^d$ this is no longer the case:
It can be controlled from a set $S$ if and only if $S$ is a thick set (thus excluding finite measure sets).

Note that the generator of the heat equation, that is, the Laplacian with suitable boundary conditions,
has purely discrete spectrum if the domain is bounded while it has purely continuous spectrum on $\RR^d$.
Thus one wonders how much this spectral dichotomy has to do with the different criteria for null-controllability.

Table~\ref{table:efficient-sets} juxtaposes necessary geometric criteria for efficient sets.
The second row for the harmonic oscillator already contains some of the results we prove.

\begin{table}[ht]
	\centering
	\begin{tabular}[t]{|m{1.75cm}|m{1.5cm}|m{3cm}|m{4.8cm}|}
		\hline
		operator & domain & spectral type & efficient set \\
		\hline\hline
		$-\Delta$ & bounded & purely discrete & necessarily bounded, finite measure \\
		\hline \rowcolor{lightgray}
		$-\Delta+|x|^2$ & $\RR^d$ & purely discrete & necessarily unbounded, may have finite measure \\
		\hline
		$-\Delta$ & $\RR^d$ & purely continuous & necessarily unbounded, infinite measure\\
		\hline
	\end{tabular}
	\vspace{1em}
	\caption{Geometric criteria for efficient sets}
	\label{table:efficient-sets}
\end{table}

The above raises the following natural question: \emph{Can one develop a better understanding of both `extremal cases' by studying interpolating models?}
The harmonic oscillator may serve as such an interpolating model:
On the one hand it is defined on the whole of $\RR^d$ but, on the other hand, it exhibits purely discrete spectrum.
Thus, it shares properties with both the Laplacian on bounded and unbounded domains.
We make this intuition more precise at the end of Subsection~\ref{ssec:main-cubes} below.

\subsection*{Outline}
The rest of this paper is organized as follows: Our main results are given in Section~\ref{sec:main-result}, where we also compare our results to previous ones.
An application in control theory for the Hermite semigroup is spelled out in Section~\ref{sec:application}.
Thereafter, in Section~\ref{sec:core-result}, we state our core result, Theorem~\ref{thm:gen}.
Section~\ref{sec:local-estimate} collects technical ingredients from previous papers for the proof of Theorem~\ref{thm:gen}, which is given in Section~\ref{sec:proof}.
Theorems~\ref{thm:main-result-fixed-cubes} and~\ref{thm:main-result-increasing-balls} are then deduced in Section~\ref{sec:coverings}.

\subsection*{Acknowledgments}
A.D.\ and A.S.\ have been partially supported by the DFG grant VE 253/10-1 entitled \emph{Quantitative unique continuation properties of elliptic PDEs with variable 2nd order coefficients and applications in control theory, Anderson localization, and photonics}.

%
%

\section{Main results and discussion}\label{sec:main-result}

We denote by $\cE_N$, $N\in\NN$, the space spanned by the Hermite functions of degree at most $N$, that is linear combinations of functions
\[
	\Phi_\alpha(x) = \prod_{j=1}^d \phi_{\alpha_j}(x_j), \quad \alpha = (\alpha_1,\dots,\alpha_d)\in\NN_0^d,\quad x=(x_1,\dots,x_d)\in\RR^d,
\]
with $|\alpha|\leq N$.
Here,
\[
	\phi_k(t) = \frac{(-1)^k}{\sqrt{2^kk!\sqrt{\pi}}}\euler^{t^2/2}\frac{\Diff{}^k}{\Diff{t}^k}\euler^{-t^2},\quad k\in\NN_0,
\]
denotes the $k$-th standard Hermite function.
Moreover, for $\rho > 0$ and $k \in (\rho\ZZ)^d$ we denote by $\Lambda_\rho(k)=k+(-\rho/2,\rho/2)^d$ the cube with sides of length $\rho$ centered at $k$.
A particular case of our general result reads as follows.

\begin{theorem}\label{thm:main-result-fixed-cubes}
	Let $S \subset \RR^d$ be measurable satisfying
	\be\label{eq:control-set-fixed-cubes}
		\frac{\abs{S \cap \Lambda_\rho(k)}}{\abs{\Lambda_\rho(k)}}
		\geq
		\gamma^{1+\abs{k}^\beta}
		\quad\text{for all}\quad
		k \in (\rho\ZZ)^d
	\ee
	with some fixed $\beta \in [0,1)$, $\rho > 0$, and $\gamma \in (0,1)$.

	Then, there is a universal constant $K \geq 1$ such that for every  $N \in \NN$ we have
	\be\label{eq:main-result-fixed-cubes}
		\norm{f}_{L^2(S)}^2
		\geq
		3\Bigl( \frac{\gamma}{K^d} \Bigr)^{K d^{5/2+\beta} (1+\rho)^2N^{(1+\beta)/2}}
		\norm{f}_{L^2(\RR^d)}^2 \quad \text{for all}\quad f \in \cE_N
		.
	\ee
\end{theorem}

For $S$ to be an efficient set, we need that the exponent of $N$ is smaller than one.
The latter is obviously satisfied with the hypothesis $\beta < 1$.
However, our more general result in Section~\ref{sec:core-result} shows that \eqref{eq:main-result-fixed-cubes} holds for all $\beta \geq 0$.

The case $\beta = 0$ in \eqref{eq:control-set-fixed-cubes} corresponds essentially to so-called \emph{thick} sets:
algorithmx.sty
\begin{definition}[Thick set]\label{def:thick}
	Let $\gamma \in (0,1]$ and $\rho>0$.
	A measurable set $S\subset \RR^d$ is called \emph{$(\gamma , \rho)$-thick} if $\abs{S \cap \Lambda_\rho(x)}\geq \gamma \abs{\Lambda_\rho(x)}$
	for all $x \in \RR^d$.
\end{definition}

If $S$ is a thick set, an uncertainty principle for Hermite functions was established in \cite[Theorem~2.1\,(iii)]{BeauchardJPS-21}.
Hence, our Theorem~\ref{thm:main-result-fixed-cubes} covers and extends this result.

However, \cite[Theorem~2.1\,(ii)]{BeauchardJPS-21} also considers sets $S$ that are not thick, but satisfy
\be\label{eq:condition-BJPS}
	\liminf_{R\to\infty}\frac{\abs{S\cap B(0,R)}}{|B(0,R)|}
	>
	0
	.
\ee
In this case \cite{BeauchardJPS-21} obtains an uncertainty relation with exponent linear in $N$, which is not enough for control theory.
By contrast, our result with $\beta \in (0,1)$ gives a class of efficient sets $S$ satisfying
\be\label{eq:condition-BJPS-counterexample}
	\liminf_{R\to\infty}\frac{\abs{S\cap B(0,R)}}{|B(0,R)|}
	=
	0
	.
\ee
In order to demonstrate this, we give an easy example.

\begin{example}
	Let $\beta,\gamma \in (0,1)$, $\rho=1$, and set
	\[
		S
		=
		\bigcup_{k\in\ZZ^d} \Lambda_{r_k}(k)\quad\text{with}\quad r_k = \frac{1}{2}\gamma^{(1+\abs{k}^\beta)/d}
		.
	\]
	Then, $S$ satisfies
	\[
		\frac{\abs{S \cap \Lambda_\rho(k)}}{\abs{\Lambda_\rho(k)}}
		\geq
		\gamma^{1 + \abs{k}^\beta}
		\quad\text{ for all}\quad
		k \in \ZZ^d
		,
	\]
	so that the hypotheses of Theorem~\ref{thm:main-result-fixed-cubes} are satisfied, and we obtain
	\bes
		\norm{f}_{L^2(S)}^2
		\geq
		3\Bigl( \frac{\gamma}{K^d} \Bigr)^{4K d^{7/2} N^{(1+\beta)/2}}
		\norm{f}_{L^2(\RR^d)}^2 \ \text{ for all } f \in \cE_N
		.
	\ees
	Note that $(1+\beta)/2< 1$,  while on the other hand, the set $S$ has finite measure since $\gamma \in (0,1)$, so that \eqref{eq:condition-BJPS-counterexample} holds.
\end{example}

\begin{remark}
	In a recent parallel development, null-controllability of the (time dependent) Schr\"odinger equation with quadratic potential
	has been studied in \cite{HuangWW-22}  and \cite{MartinPS-21}.
	To the best of our knowledge, our Theorem~\ref{thm:main-result-fixed-cubes} is the first result that treats efficient sets for the harmonic oscillator
	from which the corresponding Schr\"odinger equation is not null-controllable.
	Indeed, in \cite[Theorem~2.2]{MartinPS-21} it has been shown that condition \eqref{eq:condition-BJPS} is equivalent
	to the exact null-controllability of the Schr\"odinger equation with quadratic potential in dimension $d=1$.
	A similar condition is necessary in dimension $d\geq 2$, see \cite[Theorem~2.4]{MartinPS-21}.
\end{remark}

We have already given an example of an efficient set for the harmonic oscillator with finite measure.
In view of Table~\ref{table:efficient-sets}, we now show that a bounded set can not be an efficient set for the harmonic oscillator; see also \cite[Section 4.2.3]{Miller-08}. 
For simplicity, we here only consider the one-dimensional case.
Note however, that this does not yet prove that null-controllability of the Hermite semigroup is impossible from a bounded control set.

\begin{example}
	Let $M>0$ and $S = [-M,M]$.
	Assume that for all $N\in\NN$ the uncertainty relation
	\be\label{eq:spectral-inequality-counterexample}
		\norm{f}_{L^2(\RR)}^2
		\leq
		C\euler^{C'h(N)} \norm{f}_{L^2(S)}^2,
		\quad f\in\cE_N
		,
	\ee
	holds with some non-negative function $h$ and constants $C,C'>0$ independent from $N$.
	
	Consider $f_N(x)=x^N \euler^{-|x|^2/2}$.
	Then $f_N\in\cE_N$ and it is easy to calculate that
	\bes
		\norm{f_N}_{L^2(\RR)}^2
		=
		\Gamma\Bigl( N+\frac{1}{2}\Bigr)\quad\text{and}\quad\norm{f_N}_{L^2(S)}^2
		\leq
		\sqrt{\pi}M^{2N}
		,
	\ees
	where $\Gamma$ denotes the Gamma function.
	Plugging these into \eqref{eq:spectral-inequality-counterexample}, we derive
	\bes
		\Gamma\Bigl( N+\frac{1}{2}\Bigr)
		\leq
		\sqrt{\pi}C\euler^{C'h(N)+2N\log M},
		\quad N\in\NN
		.
	\ees
	Using Stirling's formula for $\Gamma$, this can only hold if $h(N)\geq N\log N$ for large $N$ and, in particular,
	$h$ can not be of the form $h(N)=N^\beta$ with $\beta \leq 1$.
	This also shows that \cite[Theorem~2.1\,(i)]{BeauchardJPS-21} is best possible in this situation.
\end{example}

\begin{remark}
	In dimension $d\geq 2$, a similar argument with the function $f_N(x)=x_j^N\euler^{-|x|^2/2}$, $j\in\{1,\dots,d\}$, shows that
	a set can not be efficient if it has bounded intersection with any hyperplane.
\end{remark}

In \cite{MartinPS-22}, the authors studied the case where the density of the set $S$ is considered not with respect to a fixed scale $\rho$,
but rather a variable one $\rho = \rho(x)$, which is even allowed to grow sublinearly.
Our technique allows to recover this result, while getting rid of some technical assumptions on $\rho(x)$.
More importantly, we cover again cases where the density of $S$ decays at infinity.
In order to put the corresponding result into the context of \cite[Theorem~2.1]{MartinPS-22}, we slightly change the geometry from cubes to balls.

\begin{theorem}\label{thm:main-result-increasing-balls}
	Let $\rho\colon\RR^d\to (0,\infty)$ be any function that satisfies
	\[
		\rho(x)\leq R(1+|x|^2)^{\frac{1-\eps}{2}}
		\quad\text{for all}\quad
		x \in \RR^d
	\]
	with $R>0$ and $\eps\in(0,1]$.
	Suppose that $S\subset\RR^d$ is a measurable set with
	\[
		\frac{\abs{S\cap B(x,\rho(x))}}{\abs{B(x,\rho(x))}}
		\geq
		\gamma^{1+|x|^\alpha}
		\quad\text{for all}\quad
		x \in \RR^d
	\]
	for some fixed $\alpha\in [0,\eps)$ and $\gamma \in (0,1)$.
	
	Then, there is a universal constant $K\geq 1$ such that for all $N\in\NN$ we have
	\[
		\norm{f}_{L^2(S)}^2
		\geq
		3\Bigl(\frac{\gamma}{K^d}\Bigr)^{K^{1+\alpha}d^{(11+3\alpha)/2}(1+R)^2N^{1-\frac{\eps-\alpha}{2}}}\norm{f}_{L^2(\RR^d)}^2,
		\quad f\in\cE_N
		.
	\]
\end{theorem}

The previous theorem extends the main result of \cite{MartinPS-22} in the following sense:
We do not assume that $\rho$ is $\frac{1}{2}$-Lipschitz and we allow a certain subexponential decay of the density of the set $S$.
In addition, our $\rho$ does not need to be bounded away from zero.
Note also that \cite{MartinPS-22} studies the case where $\rho$ is merely continuous, but only for certain choices of $\gamma$ and $R$.
This case is also covered by our result.
Moreover, the constant in Theorem~\ref{thm:main-result-increasing-balls} is explicit in all parameters.

Both our Theorems~\ref{thm:main-result-fixed-cubes} and~\ref{thm:main-result-increasing-balls} are particular cases of a more general result, Theorem~\ref{thm:gen} below.
How to derive these two theorems from the general result is discussed in Section~\ref{sec:coverings} below.

%
%

\section{Application to null-controllability}\label{sec:application}

In this section we show how to derive null-controllability of the parabolic evolution driven by the harmonic oscillator.
The shortest route to arrive at this claim is to spell out a criterion for observability that we take from \cite{NakicTTV-20}.
It was inspired by previous similar results, see, e.g.,~\cite{TenenbaumT-11,BeauchardPS-18}.

\begin{theorem}[{\cite[Theorem~2.8]{NakicTTV-20}}]\label{thm:obs_and_control}
	Let $H$ be a non-negative, self-adjoint operator on $L^2(\RR^d)$, and let $S \subset \RR^d$ be measurable.
	Assume that there are $d_0 > 0$, $d_1 \geq 0$ and $\zeta \in (0,1)$
	such that for all $\lambda > 0$ and all $\varphi \in L^2(\RR^d)$ we have
	\be\label{eq:UCP}
		\lVert \indic_{(-\infty,\lambda]}(H) \varphi \rVert^2
		\leq
		d_0 \euler^{d_1 \lambda^\zeta}
		\lVert \indic_{(-\infty,\lambda]}(H) \varphi \rVert_{L^2(S)}^2.
	\ee
	Then, for all $T > 0$ and all $\varphi \in L^2(\RR^d)$ we have the observability estimate
	\bes
		\bigl\lVert \euler^{-HT} \varphi \bigr\rVert^2
		\leq
		C_{\mathrm{obs}}^2
		\int_0^T  \bigl\lVert \euler^{-Ht} \varphi \bigr\rVert_{L^2(S)}^2 \Diff{t} ,
	\ees
	where $C_{\mathrm{obs}}$ satisfies
	\be\label{eq:c-obs-upper-bound}
		C_{\mathrm{obs}}^2
		=
		\frac{C_1 d_0}{T}K_1^{C_2}\exp \left(C_3\left( \frac{d_1 }{T^\zeta} \right)^{\frac{1}{1 - \zeta}} \right)
		\quad \text{with} \quad K_1 = 2 d_0  + 1 .
	\ee
	Here, $C_1,C_2,C_3>0$ are constants depending only on $\zeta$.
\end{theorem}

By canonical arguments, see, e.g.,~\cite{Zuazua-06,TucsnakW-09,Coron-07,EgidiNSTTV-20}, the conclusion of the theorem implies that for any initial value
$\varphi_0\in L^2(\RR^d)$ and any positive time $T$ we can find a control function $u \in L^2(\RR^d\times [0,T])$ such that the mild solution to
\be\label{eq:control_system}
	\dot{\varphi} + H\varphi = \indic_S u,
	\quad
	\varphi(0)=\varphi_0
	,
\ee
satisfies $\varphi(T) \equiv 0$.
This is called \emph{null-controllability} of the parabolic equation \eqref{eq:control_system}.
The above theorem also implies that in this case the so-called \emph{control cost}, that is
\eqs{
	C_T
	:=
	\sup_{\Vert \varphi_0 \Vert = 1} \min
	\Bigl\{ \lVert u \rVert_{L^2(\RR^d\times [0,T]))} \colon \euler^{-H T} \varphi_0 +  \int_0^T \euler^{-  (T-s)H} \indic_S u(\tau) \Diff{\tau} = 0 \Bigr\}
	,
}
satisfies
\[
	C_T\leq C_{\mathrm{obs}}
\]
with $C_{\mathrm{obs}}$ as in \eqref{eq:c-obs-upper-bound}.

We apply this now to the harmonic oscillator $H = -\Delta + |x|^2$, where
\[
	\Ran \indic_{(-\infty,\lambda]}(H) = \cE_{N}\quad\text{for}\quad 2N+d\leq \lambda < 2(N+1)+d.
\]
In this case, Theorems~\ref{thm:main-result-fixed-cubes} and~\ref{thm:main-result-increasing-balls} imply the following result.

\begin{theorem}\label{thm:main-null-control}
	Let $H= -\Delta + |x|^2$ denote the harmonic oscillator and let $S\subset \RR^d$ be a measurable subset such that
	\begin{enumerate}
		\item[(a)]
		there exist  $\rho>0$, $\beta \in [0,1)$ and $\gamma \in (0,1)$ with
		\bes
				\frac{\abs{S \cap \Lambda_\rho(k)}}{\abs{\Lambda_\rho(k)}}
				\geq
				\gamma^{1+\abs{k}^\beta}
				\quad\text{ for all }\quad
				k \in (\rho\ZZ)^d
		\ees
	\end{enumerate}	
	or
	\begin{enumerate}
		\item[(b)]
		there exist  $R>0$, $\eps\in(0,1]$, $\alpha\in [0,\eps)$, $\gamma \in (0,1)$ and $\rho\colon\RR^d\to (0,\infty)$ with
		\eqs{
			\rho(x)\leq R(1+|x|^2)^{\frac{1-\eps}{2}} &
			\quad\text{for all}\quad
			x \in \RR^d
		}
		and
		\eqs{
			\frac{\abs{S\cap B(x,\rho(x))}}{\abs{B(x,\rho(x))}}
			\geq
			\gamma^{1+|x|^\alpha}
			& \quad\text{for all}\quad
			x \in \RR^d
			.
		}
	\end{enumerate}	
	Then the parabolic equation \eqref{eq:control_system} is null-controllable.
\end{theorem}

\begin{example}
	Assume that $S$ is as in (b) for some $\eps\in (0,1]$ with $\alpha=\eps/2$.
	Then \eqref{eq:UCP} holds with $\zeta = 1 - \eps/4$, $d_0=1$, and
	\[
		d_1 = K^2d^7(1+R)^2 \log \Bigl(\frac{\gamma}{K^d}\Bigr).
	\]
	The upper bound \eqref{eq:c-obs-upper-bound} for the observability constant then becomes
	\[
		C_{\mathrm{obs}}^2
		\leq
		\frac{C_1 3^{C_2}}{T} \exp\Bigl( \frac{C_3K^{8/\eps}d^{28/\eps}(1+R)^{8/\eps}}{T^{4/\eps - 1}}\Bigr)
		.
	\]
	In the particular situation $\eps = 1$, this gives
	\[
		C_{\mathrm{obs}}^2
		\leq
		\frac{C_1 3^{C_2}}{T} \exp\Bigl( \frac{C_3K^{8}d^{28}(1+R)^{8}}{T^{3}}\Bigr)
		,
	\]
	which agrees (up to constants) with the observability constant we get if $S$ is as in (a) with $\beta=1/2$
	(where the parameter $R$ in (b) corresponds to the parameter $\rho$ in (a)).
\end{example}

%
%

\section{A reduction argument and the general theorem}\label{sec:core-result}

In the proofs an interplay between global and local properties of elements of $\cE_N$ plays a crucial role.
Hence, it is natural to decompose the unbounded domain $\RR^d$ into a collection of subsets.
Throughout this section, let $(Q_k)_{k \in \cJ}$ be any finite or countably infinite family of measurable subsets $Q_k \subset \RR^d$ and $\kappa \geq 1$ such that
\be\label{eq:covering-general-assumption}
	\Bigl|\RR^d \setminus \bigcup_{k \in \cJ} Q_k\Bigr|=0\qquad \text{and}\qquad
	\sum_{k \in \cJ} \indic_{Q_k} (x)
	\leq
	\kappa
	\quad \text{for all}\quad x \in\RR^d.
\ee
In other words, such a family of subsets gives an essential covering of $\RR^d$ with $\kappa$ being an upper bound for the number of overlaps between the $Q_k$.

In view of the strong localization of eigenfunctions induced by the quadratic potential of the harmonic oscillator, one expects that
most of the mass of some $f \in \cE_N$ will be concentrated in a ball centered at the origin.
However, the radius of the ball will depend on the degree $N$.
This is spelled out in the following two lemmas.
The first is taken from \cite{BeauchardJPS-21} and formulates the exponential decay of Hermite functions in terms of a weighted $L^2$-estimate.

\begin{lemma}[{\cite[Proposition~3.3]{BeauchardJPS-21}}]\label{lem:decay}
	For all $f\in\mathcal{E}_N$ we have
	\bes
		\norm{\euler^{|\cdot|^2/64d}f}_{L^2(\RR^d)}^2
		\leq
		2^{2(d+1)+N}\norm{f}_{L^2(\RR^d)}^2
		.
	\ees
\end{lemma}

With this at hand it is possible to spell out a concentration statement, making the above discussed intuition precise.
The corresponding result is similar to one in \cite{BeauchardJPS-21}.

\begin{lemma}[{cf.~\cite[Lemma~3.2]{BeauchardJPS-21}}]\label{lem:choice_Jc}
	Let $N \in \NN$. Then, with the constant $C = 32d(1+\sqrt{\log\kappa})$
	the subset $\cJ_c := \{ k\in\cJ \colon Q_k\cap B(0,CN^{1/2})\neq\emptyset\}$ satisfies
	\[
		\sum_{k \in \cJ_c^\complement} \norm{f}_{L^2(Q_k)}^2
		\leq
		\frac{1}{4}\norm{f}_{L^2(\RR^d)}^2
		\quad\text{ for all}\quad
		f \in \cE_N
		.
	\]
\end{lemma}

\begin{proof}
	For $f \in \cE_N$ and $s > 0$, we have by Lemma~\ref{lem:decay} that
	\eqs{
		\norm{f}_{L^2(\RR^d\setminus B(0,s))}^2
		& =
		\norm{\euler^{-|\cdot|^2/64d}\euler^{|\cdot|^2/64d}f}_{L^2(\RR^d\setminus B(0,s))}^2\\
		& \leq
		\euler^{-s^2/32d}2^{2(d+1)+N}\norm{f}_{L^2(\RR^d)}^2
		.
	}
	From this we easily see that
	\[
		\norm{f}_{L^2(\RR^d\setminus B(0,s))}^2
		\leq
		\frac{1}{4\kappa}\norm{f}_{L^2(\RR^d)}^2\quad\text{if}\quad  s \geq CN^{1/2}
		.
	\]
	Moreover, if $k \in \cJ_c^\complement$, then $Q_k\cap B(0,CN^{1/2})=\emptyset$. Hence,
	\[
		\sum_{k \in \cJ_c^\complement} \norm{f}_{L^2(Q_k)}^2
		\leq
		\kappa \norm{f}_{L^2(\RR^d\setminus B(0,CN^{1/2}))}^2
		\leq
		\frac{1}{4}\norm{f}_{L^2(\RR^d)}^2
		.\qedhere
	\]
\end{proof}

The above then motivates the following general hypothesis on the covering.

\begin{hypothesis}{$(\mathrm{H_N})$}\label{hypothesis}
	Let $(Q_k)_{k \in \cJ}$ be finite or countably infinite giving an essential covering of $\RR^d$
	with overlap at most $\kappa$ as in \eqref{eq:covering-general-assumption}.
	For fixed $N\in\NN$, let
	\[
		\cJ_c = \cJ_c(N) = \{ k \in \cJ \colon Q_k \cap B(0,CN^{1/2}) \neq \emptyset\}
	\]
	with $C = 32d(1 + \sqrt{\log\kappa})$. For each $k\in\cJ_c$, we suppose that
	\begin{enumerate}[(i)]
		\item
		$Q_k$ is non-empty, convex, open, and contained in a hyperrectangle with sides of length
		$l_k = (l_k^{(1)},\dots,l_k^{(d)}) \in (0,\infty)^d$ parallel to the coordinate axes;

		\item
		there is a linear bijection $\Psi_k \colon \RR^d \to \RR^d$ with
		\[
			\frac{\abs{\Psi_k(Q_k)}}{(\diam \Psi_k(Q_k))^d}
			\geq
			\eta
		\]
		for some $\eta > 0$ independent of $k\in\cJ_c$;

		\item
		we have $\norm{l_k}_1 := l_k^{(1)} + \dots + l_k^{(d)} \leq D N^{(1-\eps)/2}$ for some $\eps \in (0,1]$ and $D > 0$
		independent of $k\in\cJ_c$.
	\end{enumerate}
\end{hypothesis}

Our general result now reads as follows. Its proof is postponed to the end of Section~\ref{sec:proof} below.

\begin{theorem}\label{thm:gen}
	With fixed $N \in \NN$ assume Hypothesis~\ref{hypothesis}.
	Moreover, let $S \subset \RR^d$ be measurable satisfying
	\be\label{eq:gen_assumption}
		\frac{\abs{S \cap Q_k}}{\abs{Q_k}}
		\geq
		\gamma^{N^{\alpha/2}}
		\quad\text{ for all }\quad
		k \in \cJ_c
	\ee
	with some fixed $\alpha \geq 0$ and $\gamma \in (0,1)$.

	Then, every $f \in \cE_N$ satisfies
	\be\label{eq:gen}
		\norm{f}_{L^2(S)}^2
		\geq
		\frac{3}{\kappa} \Bigl( \frac{\eta\gamma}{24d\tau_d} \Bigr)^{7\bigl( 800\euler\sqrt{d} D(D+1) + \log(4\kappa^{1/2})\bigr) N^{1-(\eps-\alpha)/2}} \norm{f}_{L^2(\RR^d)}^2
		,
	\ee
	where $\tau_d$ denotes the Lebesgue measure of the Euclidean unit ball in $\RR^d$.
\end{theorem}

\begin{remark}
	(a) Examples for families $(Q_k)_{k\in\cJ}$ satisfying (i)--(iii) of Hypothesis~\ref{hypothesis} are discussed in Section~\ref{sec:coverings} below.

	(b) Let us emphasize that $\eta$ and $D$ in conditions (ii) and (iii), respectively, need to be uniform in $k \in \cJ_c$.

	(c) The constants $\eta$ and $D$ introduced in conditions (ii) and (iii) above are formally allowed to depend on $N$.
	However, in all applications we have in mind this will not be the case.
	Consequently, the constant in \eqref{eq:gen} then depends on $N$ only by the explicit power law $N^{1-\frac{\eps-\alpha}{2}}$.
	In this case the relevant exponent satisfies $1-\frac{\eps-\alpha}{2} < 1$ if and only if $\alpha < \eps$.
\end{remark}

%
%

\section{The local estimate and good covering sets}\label{sec:local-estimate}

On a bounded domain the following local estimate is sufficient to derive the type of uncertainty relation we are aiming at.
It goes back to Nazarov~\cite{Nazarov-94} and Kovrijkine~\cite{Kovrijkine-thesis,Kovrijkine-01}.
It is implicitly contained in several recent works such as~\cite[Section~5]{EgidiV-20},~\cite{WangWZZ-19},~\cite[Section~3.3.3]{BeauchardJPS-21},
\cite{MartinPS-22}, and~\cite[Lemma~3.5]{EgidiS-21}.
We rely here on the formulation in the last mentioned reference.

\begin{lemma}[{\cite[Lemma~3.5]{EgidiS-21}}]\label{lem:localEstimate}
	Let $N\in\NN$, $f \in \cE_N$, and let $Q \subset \RR^d$ be a non-empty bounded convex open set that is contained in a
	hyperrectangle with sides of length $l \in (0,\infty)^d$ parallel to coordinate axes.

	Then, for every measurable set $\omega \subset Q$ and every linear bijection $\Psi \colon \RR^d \to \RR^d$ we have
	\bes
		\norm{f}_{L^2(\omega)}^2
		\geq
		12 \Bigl( \frac{\abs{\Psi(\omega)}}{24d\tau_d(\diam\Psi(Q))^d} \Bigr)^{4\frac{\log M}{\log 2}+1}
			\norm{f}_{L^2(Q)}^2
	\ees
	with
	\[
		M := \frac{\sqrt{\abs{Q}}}{\norm{f}_{L^2(Q)}} \cdot \sup_{z \in Q + D_{4l}} \abs{f(z)},
	\]
	where $D_{4l}:= B(0,4l^{(1)})\times \ldots \times B(0,4l^{(d)})\subset\CC^d$ 
    denotes the anisotropic polydisc of radius $4l$ centered at the origin.
\end{lemma}

\begin{remark}\label{rem:boundEta}
	It is worth to note that the quantity $M$ in the above lemma automatically satisfies $M \geq 1$.
	Moreover, the bijection $\Psi$ can be chosen to optimize the right-hand side of
	\be\label{eq:distortion}
		\frac{\abs{\Psi(\omega)}}{(\diam\Psi(Q))^d}
		=
		\frac{\abs{\omega}}{\abs{Q}} \cdot \frac{\abs{\Psi(Q)}}{(\diam\Psi(Q))^d}
		,
	\ee
	cf.~\cite[Remark~3.6]{EgidiS-21} and also Section~\ref{sec:coverings} below.

	The proof of \cite[Lemma~3.5]{EgidiS-21} also shows that
	\[
		\abs{\Psi(Q)}
		\leq
		d\tau_d (\diam\Psi(Q))^d
		.
	\]
	In particular, the parameter $\eta$ in condition~(ii) of the covering always satisfies the upper bound $\eta \leq d\tau_d$.
\end{remark}

Throughout the remainder of this section, for fixed $N\in\NN$ we assume Hypothesis~\ref{hypothesis}.
Given a non-zero $f \in \cE_N$, let
\be\label{eq:def_Mk}
	M_k := \frac{\sqrt{\abs{Q_k}}}{\norm{f}_{L^2(Q_k)}} \cdot \sup_{z \in Q_k + D_{4l_k}} \abs{f(z)}
\ee
denote the normalized supremum from the local estimate in Lemma~\ref{lem:localEstimate} corresponding to $Q_k$.
We do not know how to guarantee an upper bound on $M_k$ for all $k$, but for `sufficiently many' $k$. In order to make this precise, we first recall the
Bernstein-type inequalities for functions in $\cE_N$ first established in \cite[Proposition~4.3\,(ii)]{BeauchardJPS-21} and
later reproduced in a slightly different form in \cite[Proposition~B.1]{EgidiS-21}.

\begin{lemma}[{\cite[Proposition~B.1]{EgidiS-21}}]\label{lem:Bernstein}
	Given $\delta > 0$, every function $f \in \cE_N$ satisfies
	\[
		\sum_{\abs{\alpha}=m} \frac{1}{\alpha!}\norm{\partial^\alpha f}_{L^2(\RR^d)}^2
		\leq
		\frac{C_B(m,N)}{m!} \norm{f}_{L^2(\RR^d)}^2
		\quad\text{ for all }\
		m \in \NN_0
	\]
	with
	\[
		C_B(m,N)
		=
		(2\delta)^{2m} \euler^{\euler/\delta^2} (m!)^2 \euler^{2\sqrt{2N+d}/\delta}
		.
	\]
\end{lemma}
The parameter $\delta > 0$ will be chosen appropriately depending on $N$ later on, see~\eqref{eq:choice_delta} below.

We use the by now well-established approach of localizing the Bernstein-type inequality on so-called good $Q_k$, which is a
modification of ideas introduced by Kovrijkine~\cite{Kovrijkine-thesis,Kovrijkine-01} and used in many works thereafter.
We rely on the form presented in~\cite[Section~3.3]{EgidiS-21}:

We say that $Q_k$ for $k \in \cJ$ is \emph{good} with respect to $f \in \cE_N$ if
\[
	\sum_{\abs{\alpha} = m} \frac{1}{\alpha!}\norm{\partial^\alpha f}_{L^2(Q_k)}^2
	\leq
	2^{m+1} \kappa \frac{C_B(m,N)}{m!} \norm{f}_{L^2(Q_k)}^2
	\quad\text{ for all }\quad
	m \in \NN,
\]
and we call $Q_k$ \emph{bad} otherwise.
It is then not difficult to show that
\be\label{eq:bad-mass}
	\sum_{k \colon Q_k\text{ bad}} \norm{f}_{L^2(Q_k)}^2
	\leq
	\frac{1}{2}\norm{f}_{L^2(\RR^d)}^2
	,
\ee
see \cite[Section~3.3]{EgidiS-21}; in particular, good $Q_k$ exist.
Moreover, for each such bounded $Q_k$ there is a point $x_{k} \in Q_k$ with
\[
	\abs{\partial^\alpha f(x_k)}
	\leq
	2^{m+1} (\kappa C_B(m,N))^{1/2} \frac{\norm{f}_{L^2(Q_k)}}{\sqrt{\abs{Q_k}}}
\]
for all $m \in \NN_0$ and all $\alpha \in \NN_0^d$ with $\abs{\alpha} = m$, see \cite[Eq.~(3.9)]{EgidiS-21}.
Using Taylor expansion around $x_k$, we now extract from the proof of \cite[Proposition~3.1]{EgidiS-21} the following result.

\begin{lemma}\label{lem:Mk}
	Let $Q_k$ be good.
	Then, the quantity $M_k$ in~\eqref{eq:def_Mk} satisfies
	\bes
		M_k
		\leq
		2\kappa^{1/2}\sum_{m \in \NN_0} C_B(m,N)^{1/2} \frac{(10 \norm{l_k}_1)^m}{m!}
		,
	\ees
	where $\norm{l_k}_1 = l_k^{(1)} + \dots + l_k^{(d)}$.
\end{lemma}

%
%

\section{Proof of Theorem~\ref{thm:gen}}\label{sec:proof}

Throughout this section we assume for some fixed $N$ Hypothesis~\ref{hypothesis}.
We need to find sufficiently many good elements $Q_k$ with $k\in\cJ_c$.
This is ensured by the following lemma and its corollary.

\begin{lemma}\label{lem:main_idea}
	Given $f \in L^2(\RR^d)$, let $\cJ_1, \cJ_2, \ldots \subset \cJ$ be such that
	\be\label{eq:main_idea_assumptions}
		\sum_{k \in \cJ_j^\complement} \norm{f}_{L^2(Q_k)}^2
		\leq
		\nu_j \norm{f}_{L^2(\RR^d)}^2
		\quad\text{ with }\quad
		\sigma:=\sum_{j \in \NN} \nu_j
		< 1
		.
	\ee
	Then,
	\[
		\norm{f}_{L^2(\RR^d)}^2
		\leq
		\frac{1}{1-\sigma}\,\sum_{k \in \bigcap_{j\in \NN} \cJ_j} \norm{f}_{L^2(Q_k)}^2
		.
	\]
\end{lemma}

\begin{proof}
	Subadditivity and \eqref{eq:main_idea_assumptions} imply
	\[
		\sum_{k \in \bigcup_{j\in \NN} \cJ_j^\complement}  \norm{f}_{L^2(Q_k)}^2
		\leq
		\sum_{j\in \NN} \sum_{k \in \cJ_c^\complement} \norm{f}_{L^2(Q_k)}^2
		\leq
		\sum_{j\in \NN}  \nu_j \norm{f}_{L^2(\RR^d)}^2
		.
	\]
	Passing to complements gives
	\[
		\sum_{k  \in \bigcap_{j\in \NN} \cJ_j} \norm{f}_{L^2(Q_k)}^2
		=
		\sum_{k \in \cJ} \norm{f}_{L^2(Q_k)}^2 - \sum_{k \in \bigcup_{j\in \NN} \cJ_j^\complement} \norm{f}_{L^2(Q_k)}^2
		\geq
		(1-\sigma) \norm{f}_{L^2(\RR^d)}^2
		,
	\]
	which proves the claim.
\end{proof}

We apply a simple version of this lemma considering only two subsets of the index set $\cJ$, namely
\be\label{eq:def_Jg}
	\cJ_g=\{k\colon Q_k\text{ good}\} \quad \text{and}\quad \cJ_c = \{ k \colon Q_k\cap B(0,CN^{1/2})\neq\emptyset\}.
\ee
In view of Lemma~\ref{lem:choice_Jc} and inequality~\eqref{eq:bad-mass}, the hypotheses of Lemma~\ref{lem:main_idea} are satisfied.
This leads to the following corollary.

\begin{corollary}\label{cor:main_corollary}
	Given $f \in \cE_N$ and $\cJ_c, \cJ_g \subset \cJ$ as in \eqref{eq:def_Jg}, we have
	\[
		\norm{f}_{L^2(\RR^d)}^2
		\leq
		4\sum_{k \in \cJ_c \cap \cJ_g} \norm{f}_{L^2(Q_k)}^2
		.
	\]
	In particular, we have $\cJ_c \cap \cJ_g \neq \emptyset$, unless $f = 0$.
\end{corollary}

We are now in position to give the proof of our main result.

\begin{proof}[Proof of Theorem~\ref{thm:gen}]
	In light of properties (i) and (ii) in Hypothesis~\ref{hypothesis}, the local estimate in Lemma~\ref{lem:localEstimate} and identity~\eqref{eq:distortion}
	yield
	\[
		\norm{f}_{L^2(Q_k\cap S)}^2
		\geq
		a_k\norm{f}_{L^2(Q_k)}^2
		\quad\text{ with}\quad
		a_k
		=
		12 \Bigl( \frac{\eta\abs{S \cap Q_k}}{24d\tau_d\abs{Q_k}} \Bigr)^{4\frac{\log M_k}{\log 2} + 1}
	\]
	for $k \in \cJ_c$, where $M_k$ is as in~\eqref{eq:def_Mk}.
	By Corollary~\ref{cor:main_corollary} we then have
	\be\label{eq:gen:min}
		\Bigl( \min_{k \in \cJ_c \cap \cJ_g} a_k \Bigr) \norm{f}_{L^2(\RR^d)}^2
		\leq
		4\sum_{k \in \cJ_c \cap \cJ_g} a_k \norm{f}_{L^2(Q_k)}^2
		\leq
		4\kappa \norm{f}_{L^2(S)}^2
		.
	\ee
	Using assumption~\eqref{eq:gen_assumption} on the set $S$, we have
	\be\label{eq:gen:ak}
		a_k
		\geq
		12\Bigl( \frac{\eta \gamma^{N^{\alpha/2}}}{24d\tau_d} \Bigr)^{4\frac{\log M_k}{\log 2}+1}
		\quad\text{ for all}\quad
		k \in \cJ_c
		.
	\ee
	Moreover, condition (iii) of Hypothesis~\ref{hypothesis} gives $\norm{l_k}_1 \leq DN^{(1-\eps)/2}$ for all $k \in \cJ_c$.
	Using Lemma~\ref{lem:Mk} and the definition of $C_B(m,N)$ from Lemma~\ref{lem:Bernstein}, we therefore get for
	all $k \in \cJ_c \cap \cJ_g$ that
	\eqs{
		M_k
		&\leq
		2\kappa^{1/2} \sum_{m \in \NN_0} C_B(m,N)^{1/2} \frac{(10 DN^{(1-\eps)/2})^m}{m!}\\
		&=
		2\kappa^{1/2} \euler^{\euler/(2\delta^2)} \euler^{\sqrt{2N+d}/\delta} \sum_{m \in \NN_0} (20\delta DN^{(1-\eps)/2})^m
		,
	}
	where $\delta > 0$ is still to be chosen. We do this as
	\be\label{eq:choice_delta}
		\delta
		=
		\bigl( 40DN^{(1-\eps)/2} \bigr)^{-1}
		,
	\ee
	so that
	\be\label{eq:change-for-scaling}
		\begin{split}
		M_k &\leq 4\kappa^{1/2}\exp(800\euler D^2 N^{1-\eps} + 40DN^{(1-\eps)/2}\sqrt{2N+d})\\
		&\leq 4\kappa^{1/2}\exp(800\euler\sqrt{d} D(D+1) N^{1-\eps/2})
		\end{split}
	\ee
	and, thus,
	\eqs{
		\log M_k
		&\leq
		\log(4\kappa^{1/2})+800\euler\sqrt{d} D(D+1) N^{1-\eps/2} \\
		&\leq
		\bigl( 800\euler\sqrt{d} D(D+1) + \log(4\kappa^{1/2}) \bigr) N^{1-\eps/2}
	}
	for all $k \in \cJ_c \cap \cJ_g$.
	Combining the latter with \eqref{eq:gen:ak}, we arrive at
	\[
		a_k
		\geq
		12\Bigl( \frac{\eta\gamma}{24d\tau_d} \Bigr)^{7\bigl( 800\euler\sqrt{d} D(D+1) + \log(4\kappa^{1/2})\bigr) N^{1-(\eps-\alpha)/2}}
	\text{ for all $k \in \cJ_c \cap \cJ_g$, }
	\]
	where we have taken into account the fact that $\eta / (24d\tau_d) \leq 1 / 24 < 1$ by
	Remark~\ref{rem:boundEta} and that $1+4/\log2 \leq 7$. In view of \eqref{eq:gen:min}, this proves the claim.
\end{proof}

\begin{remark}\label{rem:scaling}
	(a) It is worth to note that one may get a slightly sharper estimate in \eqref{eq:change-for-scaling} by
	\[
		M_k \leq 4\kappa^{1/2}\exp(800\euler D^2 N^{1-\eps}+40D\sqrt{d+2}N^{1-\eps/2}).
	\]
	This might be interesting in situations with small $D$, cf.~the discussion at the end of Subsection~\ref{ssec:main-cubes} below.
	
	(b) The only obstacle to extend the main result to an $L^p$-setting is a corresponding variant of the Bernstein inequality or a suitable replacement thereof.
\end{remark}

%
%

\section{Proof of Theorems~\ref{thm:main-result-fixed-cubes} and~\ref{thm:main-result-increasing-balls}}\label{sec:coverings}

In this section we discuss several examples satisfying Hypothesis~\ref{hypothesis} from Section~\ref{sec:core-result} and
thereby deduce Theorems~\ref{thm:main-result-fixed-cubes} and~\ref{thm:main-result-increasing-balls} from our general result Theorem~\ref{thm:gen}.
In what follows, $K\geq 1$ denotes a universal constant that can change from line to line.

\subsection{Proof of Theorem~\ref{thm:main-result-fixed-cubes}}\label{ssec:main-cubes}

Our first example addresses the situation of Theorem~\ref{thm:main-result-fixed-cubes} with $Q_k = \Lambda_\rho(k)=k + (-\rho/2 , \rho/2)^d$, $\rho > 0$.
Here, we clearly have $\cJ=(\rho\ZZ)^d$ and $\kappa = 1$, thus $C=32d$.
With $\Psi_k$ in condition (ii) being the identity, we have $\eta = 1 / d^{d/2}$.
Taking into account the asymptotic formula $\tau_d \sim (2\pi\euler/d)^{d/2} / \sqrt{d\pi}$,
we infer that $24d\tau_d/\eta \leq K^d$.
Moreover, it is easy to see that $l_k = (\rho,\dots,\rho)$ satisfies $\norm{l_k}_1 = d\rho = DN^{0}$ with $D := d\rho$. Hence, $(\Lambda_\rho(k))_{k\in\cJ}$ satisfies
Hypothesis~\ref{hypothesis} for every $N \in \NN$. Note here that both $D$ and $\eta$ are independent of $N$.

It is also not hard to verify that
\[
	\frac{|k|}{2}\leq\inf_{x\in \Lambda_\rho(k)}|x|\leq CN^{1/2}\quad\text{for all}\quad k\in\cJ_c\subset (\rho\ZZ)^d.
\]
Here the first inequality follows from the definition of $\Lambda_\rho(k)$ while the second follows from the definition of $\cJ_c$.
Finally, using these estimates, we calculate
\[
	\gamma^{1+\abs{k}^\beta}
	\geq
	\bigl( \gamma^{2^\beta} \bigr)^{1+(\abs{k}/2)^\beta}
	\geq
	\bigl( \gamma^{2^\beta} \bigr)^{1+{C}^\beta N^{\beta/2}}
	\geq
	\bigl(\gamma^{2(2C)^\beta}\bigr)^{N^{\beta/2}}
	.
\]
The claim in Theorem~\ref{thm:main-result-fixed-cubes} now follows from Theorem~\ref{thm:gen} with $\alpha = \beta$, $\eps = 1$, and
$\gamma$ replaced by $\gamma^{2(2C)^\beta}$.
The simple estimate
\[
	2\cdot (2C)^\beta\cdot 7 \bigl( 800\euler\sqrt{d} D(D+1) + \log(4)\bigr)
	\leq
	K d^{5/2+\beta}(1+\rho)^2
\]
then provides us with the particular constant in \eqref{eq:main-result-fixed-cubes}.\hfill\qedsymbol

\subsubsection*{Comparison of the harmonic oscillator and the pure Laplacian}

Here we discuss in what sense one may regard the harmonic oscillator as an interpolating model.
To this end, consider $H_t=-\Delta+t|x|^2$ for $t > 0$ and let $S\subset\RR^d$ be as in Theorem~\ref{thm:main-result-fixed-cubes}.
It is then easy to see that
\[
	H = H_1 = t^{-1/2}\cU^{-1}H_t\cU
\]
where $\cU\colon L^2(\RR^d)\to L^2(\RR^d)$ is the unitary transformation defined by
\[
	(\cU f)(x) = t^{d/8}f(t^{1/4}x),\quad x\in\RR^d.
\]
In particular,
\[
	\cU^{-1}g\in\Ran\indic_{(-\infty,t^{-1/2}\lambda]}(H)\quad\text{for all}\quad g\in\Ran\indic_{(-\infty,\lambda]}(H_t).
\]
Clearly, $\norm{g}_{L^2(\RR^d)} = \norm{\cU^{-1}g}_{L^2(\RR^d)}$ and $\norm{g}_{L^2(S)} = \norm{\cU^{-1}g}_{L^2(t^{1/4}S)}$.
Moreover, a simple calculation shows that $t^{1/4}S$ satisfies the geometric hypotheses from Theorem~\ref{thm:main-result-fixed-cubes} with $\rho$ replaced by $t^{1/4}\rho$.
Thus, using a slightly sharper estimate in the proof of Theorem~\ref{thm:gen}, see Remark~\ref{rem:scaling},
the proof of Theorem~\ref{thm:main-result-fixed-cubes} applied to $\cU^{-1}g$ and $t^{1/4}S$ yields
\[
	\norm{g}_{L^2(S)}^2
	\geq
	3\Bigl(\frac{\gamma}{K^d} \Bigr)^{Kd^\beta\bigl( 1 + d^2\rho^2 t^{1/2}+d^{3/2}\rho\lambda^{1/2}\bigr) t^{-\beta/2}\lambda^{\beta/2} }\norm{g}_{L^2(\RR^d)}^2.
\]
Here we see that the exponent on the right-hand side explodes as $t\to 0$ if $\beta > 0$.
On the other hand, if $\beta = 0$, then the limit as $t\to 0$ reproduces the known result \cite{Kovrijkine-01,Kovrijkine-thesis} for the Laplacian with thick sets.
This makes the intuition spelled out in the introduction that the harmonic oscillator may serve as an interpolating model more precise.

\subsection{Proof of Theorem~\ref{thm:main-result-increasing-balls}}\label{ssec:sectors-annuli}

In contrast to the result we have verified in the preceding subsection, the proof of Theorem~\ref{thm:main-result-increasing-balls}
starts with the construction of the family $(Q_k)_{k\in\cJ}$, as the family is not given in the formulation of the theorem.
To this end, we use the following version of the well-known Besicovitch covering theorem.

\begin{proposition}[Besicovitch]\label{prop:Besicovitch}
	If $A\subset\RR^d$ is a bounded set and $\cB$ is a family of closed balls such that each point in $A$ is the center of some ball in $\cB$,
	then there are at most countably many balls $(\overline{B}_k)\subset\cB$ such that
	\[
		\indic_A \leq \sum_k\indic_{\overline{B}_k}\leq K^d.
	\]
\end{proposition}

\begin{proof}
	The proof of Besicovitch's theorem in \cite[Theorem~2.7]{Mattila-99} establishes that the
	statement of the proposition holds with $K^d = 16^d C_d$,
	where $C_d$ is chosen such that the following implication is true:
	If $y_1,\dots,y_n\in\mathbb{S}^{d-1}$ are points with $|y_r-y_s|\geq 1$ for all $r\neq s$, then $n\leq C_d$.

	Since for such points the spherical distance $d_{\mathbb{S}^{d-1}}(y_r,y_s)$ of $y_r$ and $y_s$ can be bounded from below by
	\[
		d_{\mathbb{S}^{d-1}}(y_r,y_s) = \arccos\Bigl( 1-\frac{\abs{y_r-y_s}^2}{2}\Bigr) \geq \arccos(1-1/2) = \pi/3,
	\]
	it is easy to see that $C_d$ can be bounded by the $d$-th power of a universal constant, which proves the proposition.
\end{proof}

For given $N\in\NN$ let $A = B(0,CN^{1/2})$, where $C=32d(1+\sqrt{\log(K^d)})$.
It is then clear that the assumptions of Proposition~\ref{prop:Besicovitch} are fulfilled for $A$ and the family of balls $\cB=\{\overline{B(x,\rho(x))}\colon x\in A\}$.
This shows that there is a subset $\cI\subset\NN$ and a collection of points $(y_k)_{k\in\cI}\subset A$ such that the balls $Q_k=B(y_k,\rho(y_k))$
satisfy $\abs{A\setminus\bigcup_{k\in\cI}Q_k} = 0$.
Setting $Q_0 = \RR^d\setminus\bigcup_{k\in\cI}Q_k$, the family $(Q_k)_{k\in\cJ}$, $\cJ=\cI\cup\{0\}$, is clearly an essential covering of $\RR^d$ satisfying
\[
	\sum_{k\in\cJ}\indic_{Q_k} \leq K^d =: \kappa.
\]
Note that by construction we have $\cI=\{k\in\cJ\colon Q_k\cap A\neq \emptyset\}=\cJ_c$.

It remains to verify that $(Q_k)_{k \in \cJ}$ satisfies Hypothesis~\ref{hypothesis}:
It is easy to see that (i) is satisfied with $l_k = (2\rho(y_k),\dots,2\rho(y_k))$, and
\[
	\frac{\abs{Q_k}}{(\diam\,Q_k)^d} = \frac{\tau_d\rho(y_k)^d}{(2\rho(y_k))^d} = \frac{\tau_d}{2^d}
\]
shows that condition (ii) holds with $\eta=\tau_d/2^d$, where we have chosen $\Psi_k$ as the identity.
In particular, $(24d\tau_d)/\eta =2^d24d  \leq 48^d$.
Since $y_k\in A$ for all $k\in\cJ_c=\cI$, we have $|y_k|\leq CN^{1/2}$ and consequently
\[
	\rho(y_k)\leq 2RCN^{(1-\eps)/2}\quad\text{for all}\quad k\in\cJ_c.
\]
Combining this with the identity for $l_k$ stated above, we obtain
\[
	\norm{l_k}_1\leq 2d\rho(y_k) \leq DN^{(1-\eps)/2},\quad D=4dRC.
\]
This proves condition (iii). Thus, Hypothesis~\ref{hypothesis} is satisfied.

Using again $|y_k|\leq CN^{1/2}$ for $k\in\cJ_c$, we see that the hypothesis on the set $S$ in Theorem~\ref{thm:main-result-increasing-balls} yields
\[
	\frac{|S\cap Q_k|}{|Q_k|}\geq \gamma^{1+(CN^{1/2})^\alpha} \geq \Bigl(\gamma^{1+C^\alpha}\Bigr)^{N^{\alpha/2}}.
\]
We apply Theorem~\ref{thm:gen} with $\gamma$ replaced by $\gamma^{1+C^\alpha}$.
After adapting the constant $K$ appropriately, we have $\kappa\leq K^d$, $1+{C}^\alpha \leq  K^{1+\alpha} d^{3\alpha/2}$, and $D\leq Rd^{5/2}K$.
Hence, it is easy to see that
\[
	(1+{C}^\alpha)\cdot 7\bigl( 800\euler\sqrt{d} D(D+1) + \log(4\kappa^{1/2})\bigr)
	\leq
	K^{1+\alpha} d^{(11+3\alpha)/2} (1+R)^2
\]
and we thereby obtain the precise constant in~\eqref{thm:main-result-increasing-balls}.\hfill\qedsymbol

\begin{remark}
	Note that the Besicovitch covering theorem holds for more general convex shapes than just balls, see, e.g.,~\cite[Theorem~1.16]{Heinonen-01},
	and therefore Theorem~\ref{thm:main-result-increasing-balls} and its proof can be adapted accordingly.
\end{remark}

\newcommand{\etalchar}[1]{$^{#1}$}

\end{document}